\documentclass[12pt]{amsart}

\usepackage{hyperref}
\usepackage{color}
\usepackage{pinlabel,enumerate}
\usepackage{graphicx}
\usepackage{xypic}
\usepackage{booktabs}

\usepackage{amsmath,amsfonts,amssymb}

\theoremstyle{plain}
\newtheorem*{theorem*}{Theorem}
\newtheorem*{lemma*}{Lemma}
\newtheorem*{corollary*}{Corollary}
\newtheorem*{corollary-1}{Corollary 1}
\newtheorem*{corollary-2}{Corollary 2}

\newtheorem*{proposition*}{Proposition}
\newtheorem*{proposition-1}{Proposition 1}
\newtheorem*{proposition-2}{Proposition 2}
\newtheorem*{proposition-3}{Proposition 3}

\newtheorem{conjecture*}{Conjecture}
\newtheorem{theorem}{Theorem}[section]

\newtheorem{lemma}[theorem]{Lemma}

\theoremstyle{remark}

\newtheorem*{remark}{Remark}

\theoremstyle{definition}

  \def\Q{\mathbb{Q}} \def\F{\mathbb{F}} \def\Z{\mathbb{Z}} \def\R{\mathbb{R}}  
\def\N{\mathbb{N}}    
 \def\a{\alpha}  \def\tor{\operatorname{Tor}} \def\bp{\begin{pmatrix}}
 \def\ep{\end{pmatrix}} \def\bn{\begin{enumerate}} 
   \def\en{\end{enumerate}}
\def\ba{\begin{array}} \def\ea{\end{array}}  
   \def\a{\alpha}  \def\ti{\tilde}

\def\ker{\operatorname{Ker}}\def\be{\begin{equation}} \def\ee{\end{equation}}

   \def\eps{\epsilon}
 \def\dim{\operatorname{dim}}

    \def\fr12{\frac{1}{2}} \def\z12{\Z[\fr12]}

\def\G{\Gamma}

\def\ti{\tilde}
\def\G{\Gamma}

\renewcommand\epsilon{\varepsilon}
\DeclareMathAlphabet{\mathbf}{OML}{cmm}{b}{it}

\numberwithin{equation}{section}

\begin{document}

\title[Growth of Betti numbers and ranks of $3$-manifold groups]{A note on the growth of Betti numbers and ranks of $3$-manifold groups}

\author{Stefan Friedl}
\address{Mathematisches Institut\\ Universit\"at zu K\"oln\\   Germany}
\email{sfriedl@gmail.com}

\begin{abstract}
Let $N$ be an irreducible, compact 3-manifold with empty or toroidal boundary which is not a closed graph manifold.
We show that it follows from the work  of Agol, Kahn-Markovic and Przytycki-Wise  that  $\pi_1(N)$  admits a cofinal filtration with `fast' growth of Betti numbers as well as  a
cofinal filtration of $\pi_1(N)$  with `slow' growth of ranks.
%
\end{abstract}

\maketitle
\section{Introduction}

 A \emph{filtration of a group $\pi$}
is a sequence $\{\pi_i\}_{i\in \N}$ of  finite index subgroups of $\pi$ such that $\pi_{i+1}\subset \pi_i$ for every   $i$.
We say that a filtration is \emph{cofinal} if  $\cap_{i\in \N}\pi_i$ is trivial, we call it \emph{normal} if $\pi_i \lhd \pi$ for every $i$,
and we say  it is \emph{almost normal} if there exists a $k$ such that $\pi_i\lhd \pi_k$ for every $i\geq k$.
A group which admits a cofinal normal filtration is called \emph{residually finite}.

Given a filtration $\{\pi_i\}_{i\in \N}$ of a group $\pi$ it is of interest to  study how the following measures of `complexity'  grow:
\bn
\item the first Betti number $b_1(\pi_i)=\dim H_1(\pi_i;\Q)$,
\item the $\F_p$-Betti numbers $b_1(\pi_i,\F_p)=\dim H_1(\pi_i;\F_p)$,
\item the rank $d(\pi_i)$, i.e. the minimal size of a generating set,
\item the order of $\tor H_1(\pi_i;\Z)$.
\en
Such growth functions have been studied for 3-manifold groups by many authors over the years.
 We refer to \cite{CE10,CW03,De10,EL12,Gi10,GS91,KMT03,La09,La11,Le10,Lu94,LL95,KS12,Ra10,Ri90,ShW92,SiW02a,SiW02b,Wa09} for a sample of results in this direction.
It is clear that given any group $\pi$ we have $d(\pi)\geq b_1(\pi)$, i.e. given a filtration the ranks grow at least as fast as the Betti numbers.

Now let $N$ be a 3-manifold. Throughout this paper we will use the following convention:
a 3-manifold  will always be assumed to be connected, compact, orientable and irreducible with empty or toroidal boundary.
By \cite{He87} the group  $\pi_1(N)$ is residually finite.
In this paper we are interested in how fast Betti numbers can grow in a cofinal filtration of $\pi_1(N)$  and how slowly the ranks can grow in a cofinal filtration of $\pi_1(N)$.

First note that given any cofinal normal filtration $\{\pi_i\}_{i\in \N}$  of $\pi=\pi_1(N)$ it follows from  the work of L\"uck \cite[Theorem~0.1]{Lu94} and  Lott and L\"uck \cite[Theorem~0.1]{LL95} that
\be \label{equ:ll} \lim_{i\to \infty} \frac{1}{[\pi:\pi_i]} b_1(\pi_i)=0,\ee
i.e. the first Betti number grows sublinearly.
The same equality also holds for almost normal cofinal filtrations of $\pi_1(N)$ if we apply the aforementioned results to an appropriate finite cover of $N$.

\begin{remark}
Note that (\ref{equ:ll}) does not necessarily hold for cofinal filtrations of $\pi_1(N)$ which are not almost normal.
In fact Gir\~{a}o \cite{Gi10} (see proof of \cite[Theorem~3.1]{Gi10}) gives an
example of a cusped hyperbolic 3-manifold together with a cofinal filtration of  $\{\pi_i\}_{i\in \N}$  of $\pi=\pi_1(N)$ such that
\[ \lim_{i\to \infty} \frac{1}{[\pi:\pi_i]} b_1(\pi_i)>0.\]
\end{remark}

It is an interesting question how quickly  $\frac{1}{[\pi:\pi_i]} b_1(\pi_i)$ converges to zero,
and to what degree the convergence depends on the choice of normal cofinal filtration of $\pi=\pi_1(N)$.
This question for example was recently studied by Kionke and Schwermer \cite{KS12}.

We will use recent work of  Agol \cite{Ag12} (which in turn builds on work of Kahn-Markovic \cite{KM12} and Wise \cite{Wi12})
to prove the following theorem which says that `most' 3-manifolds admit cofinal filtrations with `fast' sublinear growth of first Betti numbers.

\begin{theorem}\label{mainthm1}
Let  $N\ne S^1\times D^2$ and $N\ne T^2\times I$ be a 3-manifold which is neither  spherical nor covered by a torus bundle.
Then the following hold:
\bn
\item Given any function  $f\colon \N\to \R_{\geq 0}$ such that
\[ \lim_{n\to \infty} \frac{f(n)}{n}=0\]
  there exists
an almost normal cofinal filtration $\{\pi_i\}_{i\in \N}$ of  $\pi$
such that
\[ b_1(\pi_i) \geq f([\pi:\pi_i])  \mbox{ for every }i\in \N.\]
\item There exists
a normal cofinal filtration  $\{\pi_i\}_{i\in \N}$  of $\pi=\pi_1(N)$
and an $\eps\in (0,1)$
such that
\[ b_1(\pi_i)\geq [\pi:\pi_i]^\eps \mbox{ for every }i\in \N.\]
\en
\end{theorem}

We now turn to the construction of cofinal filtrations with `slow' growth of ranks.
First note that if $H$ is a finite index subgroup of a finitely generated group $G$, then it follows from the Reidemeister-Schreier method (see e.g. \cite[Corollary~2.7.1]{MKS76})
that
\[ d(H)\leq [G:H]\cdot (d(G)-1)+1\leq [G:H]\cdot d(G).\]
In particular if $\{\pi_i\}_{i\in \N}$ is a cofinal filtration  of a group $\pi$, then
\[ \frac{1}{[\pi:\pi_i]} d(\pi_i)\leq d(\pi) \mbox{ for every }i.\]
Put differently, the rank grows at most linearly with the degree.

We will again use the recent work of  Agol, Kahn-Markovic and Wise together with work of Przytycki-Wise \cite{PW12}
to prove the following theorem which says that `most' 3-manifolds admit cofinal filtrations with `slow'  growth of ranks.

\begin{theorem}\label{mainthm2}
Let $N$ be a 3-manifold which is not a closed graph manifold.
\bn
\item
Given any  function  $f\colon \N\to \R_{\geq 0}$ with
\[ \lim_{n\to \infty} f(n)=\infty\]
 there exists
an almost normal cofinal filtration  $\{\pi_i\}_{i\in \N}$ of $\pi$
such that
\[ d(\pi_i)\leq f([\pi:\pi_i]) \mbox{ for every }i\in \N.\]
\item There exists
a normal cofinal filtration  $\{\pi_i\}_{i\in \N}$ of $\pi_1(N)$
and an $\eps\in (0,1)$
such that
\[d(\pi_i)\leq [\pi:\pi_i]^{\eps}\mbox{ for every }i\in \N.\]
\en
\end{theorem}

\subsection*{Acknowledgment}
We wish to thank Jack Button and Wolfgang L\"uck for helpful conversations.
We are also grateful to the anonymous referees for carefully reading an earlier version of the paper
and for giving helpful feedback. 

\section{Proofs}

\subsection{3-manifold groups}

The world of 3-manifold topology was shaken up considerably by the recent breakthroughs due to Agol, Kahn-Markovic, Przytycki-Wise  and Wise.
In particular the following is a consequence of these recent results:

\begin{theorem}\label{thm:akmw}
Let $N$ be a 3-manifold.
\bn
\item Suppose that $N\ne S^1\times D^2$ and $N\ne T^2\times I$ and suppose that   $N$  is neither spherical nor covered by a torus bundle.
Then $\pi_1(N)$ is large, i.e. $\pi_1(N)$ contains a finite index subgroup which admits an epimorphism onto a non-cyclic free group.
\item Suppose that $N$ is  not a closed graph manifold.
 Then $N$ is virtually fibered, i.e. $N$  admits a finite index cover which fibers over $S^1$.
\en
\end{theorem}

The first statement is a consequence of the `Virtually Compact Special Theorem' of Agol \cite{Ag12} (building on work of Kahn-Markovic \cite{KM12} and Wise \cite{Wi12})
and older work of Kojima \cite{Ko87} and Luecke \cite{Lue88}. The second statement is also a  consequence of the `Virtually Compact Special Theorem' together with further work of  Agol \cite{Ag08}
and Przytycki-Wise \cite{PW12}. The fact that graph manifolds with boundary are fibered follows from earlier work of
Wang--Yu \cite{WY97} (see also \cite{Li11,PW11}).
We refer to the survey paper \cite{AFW12} for details and how this theorem follows precisely from the aforementioned papers.

\subsection{Growth of the first Betti number of large groups} \label{section:largeb1growth}


In this section we will several times make use of the basic fact that if $\varphi\colon G\to H$ is a group homomorphism with finite cokernel,  then a transfer argument shows
that $H_1(G;\Q)\to H_1(H;\Q)$ is surjective, and therefore $b_1(G)\geq b_1(H)$.
We start out with the following lemma.

\begin{lemma}\label{lem:fastb1growth}
Let $\G$ be a residually finite group which admits an epimorphism $\a\colon \G \to F$ onto a non-cyclic free group.
Let $g\colon \N\to \R_{\geq 0}$ be a function such that
\[ \lim_{n\to \infty} \frac{g(n)}{n}=0.\]
Then  there exists
a normal cofinal filtration $\{\G_i\}_{i\in \N}$ of $\G$
such that
\[ b_1(\G_i)\geq g([\G:\G_i])  \mbox{ for every }i\in \N.\]
\end{lemma}

\begin{proof}
Let $\G$ be a residually finite group which admits an epimorphism $\a\colon \G \to F$ onto a non-cyclic free group.
Let $g\colon \N\to \R_{\geq 0}$ be a function such that $\lim_{n\to \infty} \frac{g(n)}{n}=0$.
After possibly replacing $g$ by
\[ n\mapsto \max\{g(1),\dots,g(n)\}\]
we can and will assume that $g$ is monotonically increasing.
\footnote{Note that if $\lim_{n\to \infty} \frac{g(n)}{n}=0$ and if we set $f(n):=\max\{g(1),\dots,g(n)\}$,
then $\lim_{n\to \infty} \frac{f(n)}{n}=0$ as well. Indeed, let $\eps>0$.
By assumption there exists an ${N}$ such that $\frac{g(n)}{n}<\eps$ for all $n\geq {N}$.
We now let $M$ be any integer greater than ${N},\frac{2}{\eps}g(1),\dots,\frac{2}{\eps}g({N}-1)$.
For every $n\geq M $ we then  have
\[\ba{rcl}  \frac{1}{n} f(n)&=& \max\{\frac{1}{n}g(1),\dots,\frac{1}{n}g({N}-1),\frac{1}{n}g({N}),\dots,\frac{1}{n}g(M)\}\\[2mm]
&\leq& \max\{\frac{1}{M}g(1),\dots,\frac{1}{M}g({N}-1),\frac{1}{{N}}g({N}),\dots,\frac{1}{M}g(M)\}<\eps.\ea\]}

Let $\{G_i\}_{i\in \N}$ be any normal cofinal filtration  of $\G$.
We denote the projection maps $\G\to \G/G_i$, $i\in \N$, by $\rho_i$.
We write $d_i:=[\G:G_i]$, $i\in \N$. We pick an epimorphism $\phi\colon F\to \Z$  and given $n\in \N$ we denote by
$\phi_n\colon F\xrightarrow{\phi}\Z\to \Z/n$ the canonical projection. We also write $\psi_n=\phi_n\circ \a$.

Since $ \lim_{n\to \infty} \frac{g(n)}{n}=0$
we can iteratively pick $n_i\in \N$ with
\[ \frac{g(n_id_i)}{n_id_i}< \frac{1}{d_i}, \mbox{ i.e. such that } g(n_id_i)<n_i\]
 and such that $n_{i+1}|n_i$ if $i>1$. We now define
\[ \G_i:=\ker\{ \rho_i\times \psi_{n_i}\colon \G\to \G/G_i\times \Z/n_i\}.\]
Note that $n_id_i\geq [\G:\G_i]$ and note that $\{\G_i\}_{i\in \N}$ is a cofinal normal filtration of $\G$.
Given any $i\in \N$ we then have
\[ \ba{rcl} \frac{1}{g([\G:\G_i])} b_1(\G_i)&\geq& \frac{1}{g(n_id_i)}b_1(\G_i)\\
&\geq & \frac{1}{n_i}b_1(\ker\{ \rho_i\times \psi_{n_i}\colon \G\to \G/G_i\times \Z/n_i\})\\[2mm]
&\geq & \frac{1}{n_i}b_1(\ker\{  \phi_{n_i}:F\to \Z/n_i\})\\
&=&\frac{1}{n_i}(n_ib_1(F)-1)\geq 1.\ea \]
\end{proof}

We are now in a position to prove Theorem \ref{mainthm1}.

\begin{proof}[Proof of Theorem \ref{mainthm1}]
Let  $N\ne S^1\times D^2$ and $N\ne T^2\times I$ be a 3-manifold which is neither  spherical nor covered by a torus bundle.
By Theorem  \ref{thm:akmw} (1) the group $\pi=\pi_1(N)$ is large, i.e. it admits a finite index subgroup $\G$ which surjects onto a non-cyclic free group. Since this property is preserved by going to finite index subgroups we can assume that $\G$ is a normal subgroup of $\pi$.
We write $k=[\pi:\G]$.
\bn
\item Let $f\colon \N\to \R_{\geq 0}$ be a function with $\lim_{n\to \infty} \frac{f(n)}{n}=0$.
After possibly replacing $f$ by
\[ n\mapsto n\sup \left\{ \frac{f(n)}{n},\frac{f(n+1)}{n+1},\dots \right\}\]
we can and will assume that $\frac{f(n)}{n}$ is monotonically decreasing.

We apply  Lemma \ref{lem:fastb1growth} to $\G$ and the function $g(n)=kf(n)$
and we denote by $\{\G_i\}_{i\in \N}$ the resulting  cofinal normal filtration of $\G$.
Note that $\{\G_i\}_{i\in \N}$ is a cofinal almost normal filtration of $\pi$, and that
\[ \ba{rcl} b_1(\G_i)&\geq & f([\G:\G_i])[\pi:\G]\\
&=&\frac{f([\G:\G_i])}{[\G:\G_i]}[\G:\G_i][\pi:\G]\\
&\geq& \frac{f([\pi:\G_i])}{[\pi:\G_i]}[\G:\G_i][\pi:\G]=f([\pi:\G_i]).\ea \]
\item
By Lemma \ref{lem:fastb1growth} there exists a cofinal normal filtration $\{\G_i\}_{i\in \N}$ of $\G$  such that
\[ b_1(\G_i) \geq k^{\frac{1}{2k}}\sqrt{[\G:\G_i]} \mbox{ for every }i\in \N.\]
We pick a complete set of representatives $a_1,\dots,a_k$ for $\pi/\G$. Given $i\in \N$ we define
\[ \pi_i:=\bigcap\limits_{j=1}^k a_j\G_ia_j^{-1}.\]
Note that $\{\pi_i\}_{i\in \N}$ is  a normal cofinal filtration of $\pi$.
Also note that
\[ \pi_i=\ker\{ \G\to \G/a_1\G_ia_1^{-1} \times \dots \times  \G/a_k\G_ia_k^{-1}\}.\]
It thus follows that
\[ [\pi:\pi_i]=[\pi:\G]\cdot [\G:\pi_i]\leq [\pi:\G]\cdot [\G:\G_i]^k=k\cdot [\G:\G_i]^k.\]
Finally note that $b_1(\pi_i)\geq b_1(\G_i)$, we thus see that
for every  $i$ we have
\[ b_1(\pi_i)\geq b_1(\G_i)\geq k^{\frac{1}{2k}}\sqrt{[\G:\G_i]} \geq [\pi:\pi_i]^{\frac{1}{2k}}.\]
\en
\end{proof}

\begin{remark}
It seems unlikely   that one can turn the almost normal sequence of Theorem \ref{mainthm1} (1)
into a normal sequence without paying a price. For example consider the group
\[ \pi= \Z/2 \ltimes (F\times F)\]
where $F$ is a free non-cyclic group and $1\in \Z/2$ acts by commuting the two copies of $F$.
If we apply the principle of the proof of Theorem \ref{mainthm1} (1)
to $\G=F\times F$ and $\a\colon F\times F\to F$ the projection on the first factor and $\G_n:=\ker\{F\times F\to F\to \Z/n\}$,
then if we normalize these groups we really take the kernel $\ker\{F\times F\to F\to \Z/n\times \Z/n\}$
but now the growth of the Betti numbers is sublinear (in fact it grows with the square root of the index).
\end{remark}

\subsection{Growth of the rank of virtually fibered 3-manifolds} \label{section:slowgrowth}

In the following we mean by a surface group $G$ the fundamental group of a compact orientable surface.
We will make use of the following two facts:
\bn
\item For any surface group $G$ we have $b_1(G)=d(G)$.
\item If $H$ is a finite index subgroup of a surface group $G$, then an  Euler characteristic argument shows that
$ b_1(H)\leq l\cdot b_1(G)$.
\en
We can now formulate and prove the following lemma.

\begin{lemma}\label{lem:dsemidirect}
Let $\G=\Z\ltimes G$
be the semidirect product of $\Z$ with a surface group $G$.
Let $f\colon \N\to \R_{\geq 0}$ be a function with $ \lim_{n\to \infty} f(n)=\infty$.
Then  there exists
a normal cofinal filtration  $\{\G_i\}_{i\in \N}$ of $\G$
such that
\[ d(\G_i)\leq f([\G:\G_i]) \mbox{ for every }i\in \N.\]
\end{lemma}

\begin{proof}
Let $G$ be a surface group. We write $r=b_1(G)$.
Note that surface groups  are residually finite,
in particular there exists a cofinal filtration $\{G_i\}_{i\in \N}$  of $G$ by characteristic finite index subgroups of $G$. (Recall that a subgroup of $G$ is called characteristic if it is preserved by every automorphism of $G$.)
We write $d_i:=[G:G_i]$, $i\in \N$.

We denote by $\phi\colon \G=\Z\ltimes G\to \Z$ the projection onto the first factor and given $n\in \N$ we denote
by $\phi_n\colon \G=\Z\ltimes G\to \Z/n$ the composition of $\phi$ with the surjection onto $\Z/n$.
Since $\lim_{n\to \infty} f(n)=\infty$ we can  iteratively pick $n_i\in \N$ such that
\[ f(n_id_i)\geq 1+d_ir\]
and such that $n_i|n_{i+1}$ for $i>1$.
We then define $ \G_i:=n_i\Z\ltimes G_i$.
Note that $\G_i, i\in \N$ is normal in $\G=\Z\ltimes G$ since $G_i\subset G$ is characteristic.
In particular the $\{\G_i\}_{i\in \N}$ form a normal cofinal filtration of $\G$.
It now follows that
\[ \ba{rcl}  d(\G_i)=d(n_i\Z\ltimes G_i)\leq 1+d(G_i)&=&1+b_1(G_i) \\
&\leq& 1+d_ir\\
&\leq& f(n_id_i)= f([\G:\G_i]).\ea\]
\end{proof}

We are now in a position to prove Theorem \ref{mainthm2}.

\begin{proof}[Proof of Theorem \ref{mainthm2}]
Let $N$ be a 3-manifold which is not a closed graph manifold. We write $\pi=\pi_1(N)$.
By Theorem \ref{thm:akmw} (2) there exists a finite cover $\ti{N}$ which fibers over $S^1$, i.e. $\G:=\pi_1(\ti{N})\cong \Z\ltimes G$, where $G$ is a surface group.
Since finite covers of fibered 3-manifolds are again fibered, we can  assume that $\G:=\pi_1(\ti{N})$ is a normal subgroup of $\pi$.
\bn
\item
Let $g\colon \N\to \R_{\geq 0}$ be a function with $ \lim_{n\to \infty} g(n)=\infty$.
We then apply Lemma \ref{lem:dsemidirect} to $\G=\Z\ltimes G$ and $f(n):=\frac{1}{[\pi:\G]}g(n)$.
The resulting filtration is an almost normal cofinal filtration of $\pi$ with the desired property.
\item

By Lemma \ref{lem:dsemidirect}   there exists
a normal cofinal filtration $\{\G_i\}_{i\in \N}$ of $\G$
 such that
\[ d(\G_i)\leq  [\G:\G_i]^{\frac{1}{2}}\mbox{ for all $i$.}\]
Given $i\in \N$ we write $n_i:=[\G:\G_i]$.
We now denote by $a_1,\dots,a_k$ a complete set of representatives of $\pi/\G$.
Given any $i\in \N$ we define
\[ \pi_i:=\bigcap\limits_{j=1}^k a_j\G_ia_j^{-1}\subset \G_i.\]
Note that $\{\pi_i\}_{i\in \N}$ is now a  normal cofinal filtration of $\pi$.
Given $i\in \N$ we write $s_i:=[\G_i:\pi_i]$. Note that $n_i\cdot s_i=[\G:\G_i]\cdot [\G_i:\pi_i]\leq n_i^k$.
We thus see that $s_i\leq n_i^{k-1}$.
Using this observation we obtain that
\[ \ba{rcl} d(\pi_i) \leq [\G_i:\pi_i]\cdot d(\G_i)&=&s_i \cdot  n_i^{\frac{1}{2}}\\[1mm]
&=& s_i^{\frac{2k-1}{2k}}s_i^{\frac{1}{2k}}\cdot n_i^{\frac{1}{2}}
\leq s_i^{\frac{2k-1}{2k}}\cdot n_i^{\frac{k-1}{2k}}n_i^{\frac{1}{2}}\\[1mm]
&=&  s_i^{\frac{2k-1}{2k}}n_i^{\frac{2k-1}{2k}}=  k^{-\frac{2k-1}{2k}} (s_in_ik)^{\frac{2k-1}{2k}}\\
&=& k^{-\frac{2k-1}{2k}}\cdot  [\pi:\pi_i]^{\frac{2k-1}{2k}}.\ea \]
It follows that the sequence $\{\pi_i\}_{i\in \N}$  together with $\eps=\frac{2k-1}{2k}$ has the desired properties.
\en
\end{proof}

\end{document}